\documentclass[11pt]{amsart}
\usepackage{color}
\usepackage{amsmath}
\usepackage{amssymb}
\usepackage{amsfonts}
\usepackage[latin1]{inputenc}

\DeclareMathSymbol{\subsetneqq}{\mathbin}{AMSb}{36}

\newtheorem{theorem}{Théorème}
\theoremstyle{plain}

\newtheorem{corollary}{Corollary}

\newtheorem{lemma}{Lemme}

\newtheorem{remark}{Remarque}

\numberwithin{equation}{section}
\input{tcilatex}

\begin{document}
\title[contrôle de l'explosion des solutions des équations de Navier-Stokes]{%
Rôle de l'espace de Besov $\mathbf{B}_{\infty }^{-1,\infty }$dans le contrô%
le de l'explosion éventuelle en temps fini des solutions régulières des é%
quations de Navier-Stokes}
\author{Ramzi May}
\address{Département des mathématiques, Université d'Evry, boulevard F.
Mitterrand, 91025 Evry Cedex, France}
\email{Ramzi.May@fsb.rnu.tn}
\maketitle

\noindent \textbf{Résumé. }Soit $u\in C([0,T^{\ast }[;L^{n}(\mathbb{R}%
^{n})^{n})$ une solution maximale des équations de Navier-Stokes. Nous
montrons que $u$ est $C^{\infty }$ sur $]0,T^{\ast }[\times \mathbb{R}^{n}$
et qu'il existe une constante $\varepsilon _{\ast }>0,$ qui ne dépend que de
$n,$ telle que si $T^{\ast }<\infty $ alors, pour toute $\omega \in S(%
\mathbb{R}^{n})^{n},$ on a $\overline{\lim_{t\rightarrow T^{\ast }}}\left\|
u(t)-\omega \right\| _{\mathbf{B}_{\infty }^{-1,\infty }}\geq \varepsilon
_{\ast }.\smallskip $

\begin{center}
\textbf{The role of the Besov space }$\mathbf{B}_{\infty }^{-1,\infty }$%
\textbf{\ in the control of the eventual explosion in finite time of the
regular solutions of the Navier-Stokes equations }
\end{center}

\smallskip

\noindent \textbf{Abstract.} Let $u\in C([0,T^{\ast }[;L^{n}(\mathbb{R}%
^{n})^{n})$ be a maximal solution of the Navier-Stokes equations. We prove
that $u$ is $C^{\infty }$ on $]0,T^{\ast }[\times \mathbb{R}^{n}$ and there
exists a constant $\varepsilon _{\ast }>0$, which depends only on $n,$ such
that if $T^{\ast }$ is finite then, for all $\omega \in S(\mathbb{R}%
^{n})^{n},$ we have $\overline{\lim_{t\rightarrow T^{\ast }}}\left\Vert
u(t)-\omega \right\Vert _{\mathbf{B}_{\infty }^{-1,\infty }}\geq \varepsilon
_{\ast }.$

\section{Introduction et énoncé des résultats}

Nous utilisons dans cette note les notations suivantes: on désigne par $n $
un entier fixe supérieur à $3.$ Tous les espaces fonctionnels considérés
sont définis sur l'espace $\mathbb{R}^{n}.$ Si $(X,\left\| .\right\| )$ est
un espace de Banach tel que $S(\mathbb{R}^{n})\hookrightarrow
X\hookrightarrow S^{\prime }(\mathbb{R}^{n}),$ on note par $\mathbf{X}$
l'espace produit $X^{n},$ on pose $\mathbf{X}_{\sigma }=\{f\in \mathbf{X;\,}%
div(f)=0\}$ et on désigne par $\mathbf{\tilde{X}}$ l'adhérence de $S(\mathbb{%
R}^{n})^{n}$ dans $\mathbf{X.}$

Soient $u_{0}\in \mathbf{L}_{\sigma }^{n}$ et $u\in C([0,T^{\ast }[;\mathbf{L%
}_{\sigma }^{n})$ la solution maximale des équations intégrales de
Navier-Stokes associées à la donnée initale $u_{0}$%
\begin{equation}
u(t)=e^{t\Delta }u_{0}+\mathbb{L}(\mathbb{P}\nabla .(u\otimes u))(t),
\tag{NSI}
\end{equation}
où $\mathbb{P}$ est le projecteur de Leray et $\mathbb{L}$ est l'opérateur
linéaire défini par
\begin{equation}
\mathbb{L(}f\mathbb{)(}t\mathbb{)=-}\int_{0}^{t}e^{(t-s)\Delta }f(s)ds.
\label{op}
\end{equation}
Pour l'existence et l'unicité de la solution $u,$ on pourra consulter les réf%
érences (\cite{Fur}, \cite{Lem}, \cite{Mey}). En ce qui concerne la régularit%
é de $u,$ P.G.Lemarié-Rieusset \cite{Lem} a montré, en utilisant le critère
de Caffarelli, Kohn et Nirenberg, que la solution $u $ appartient à l'espace
$C(]0,T^{\ast }[,\mathbf{L}^{\infty })$ et qu'elle est, par conséquent, de
classe $C^{\infty }$ sur l'ouvert $Q_{T^{\ast }}=]0,T^{\ast }[\times \mathbb{%
R}^{n}.$ Nous présentons, dans ce papier, une autre démonstration é%
lementaire et directe de la régularité de $u.$

Supposons, dorénavant, que notre solution $u$ explose en temps fini i.e $%
T^{\ast }<\infty $, (nous ne savons pas, jusqu'à présent, si un tel phénomé%
ne est possible ou non). Notre objectif principal est d'étudier le
comportement de $u(t)$ au voisinage de $T^{\ast }.$ Rappelons que Y. Giga
\cite{Gig} a prouvé que les normes $\left\| u(t)\right\| _{p},$ $n<p\leq
\infty ,$ tendent vers l'infini avec une vitesse supérieure à $C_{p}(T^{\ast
}-t)^{\frac{1}{2}(\frac{n}{p}-1)}.$ Dans le cas limite où $p=n,$ H. Sohr et
W.Von Wahl [10] ont montré que $u$ ne peut pas être uniformement continue
sur $[0,T^{\ast }[$ à valeurs dans l'espac $\mathbf{L}^{n}.$ H. Kozono et H.
Sohr \cite{Koz} ont précisé ce résultat en montrant l'existence d'une
constante $\varepsilon _{KS}>0$, qui ne dépend que de $n,$ telle que si $%
\lim_{t\rightarrow T^{\ast }}u(t)=u^{\ast }$ dans $\mathbf{L}^{n}$ faible,
alors $\overline{\lim_{t\rightarrow T^{\ast }}}\left\| u(t)\right\|
_{n}^{n}-\left\| u^{\ast }\right\| _{n}^{n}\geq \varepsilon _{KS}.$ Comme
conséquence, ils ont prouvé que $u$ n'appartient pas à l'espace $%
BV([0,T^{\ast }[,\mathbf{L}^{n}).$ Récemment, L. Escauriaza, G. Seregin et
V. \v{S}verák \cite{Esc} ont montré, qu'en trois dimensions d'espace ($n=3$%
), si la solution $u$ appartient en plus à l'espace d'énegie de Leray-Hopf $%
\mathcal{L}_{T^{\ast }}=L_{T^{\ast }}^{\infty }(\mathbf{L}^{2})\cap
L_{T^{\ast }}^{2}(\mathbf{H}^{1}),$ alors $\overline{\lim_{t\rightarrow
T^{\ast }}}\left\| u(t)\right\| _{3}=\infty .$ Dans cette note, nous dé%
montrons le théorème suivant qui précise le comportement de $u(t)$ dans
l'espace limite de Besov $\mathbf{B}_{\infty }^{-1,\infty }$ (rappelons que $%
\mathbf{L}^{n}\hookrightarrow \mathbf{\tilde{B}}_{\infty }^{-1,\infty }).$
Ce résultat est fort utile dans l'étude de la régularité des solutions
faibles des équations de Navier-Stokes \cite{May}.\medskip

\begin{theorem}
\label{th1}\textit{Il existe une constante }$\varepsilon _{\ast }>0$\textit{%
\ qui ne dépend que de }$n$ \textit{telle que, pour toute }$\omega \in
\mathbf{\tilde{B}}_{\infty }^{-1,\infty },$\textit{\ on a}
\begin{equation}
\overline{\lim_{t\rightarrow T^{\ast }}}\left\Vert u(t)-\omega \right\Vert _{%
\mathbf{\tilde{B}}_{\infty }^{-1,\infty }}\geq \varepsilon _{\ast }.
\label{E1}
\end{equation}
\end{theorem}

\begin{remark}
nous montrons dans \cite{May} que ce résultat de persistance reste vrai
lorsqu'on remplace l'espace $\mathbf{L}^{n}$ par d'autres espaces
fonctionels tels que les espaces de Lebesgue $\mathbf{L}^{p}$ (avec $p>n$)
ou l'espace de Sobolev $\mathbf{H}^{\frac{d}{2}-1}.$
\end{remark}

Une conséquence immédiate de ce Théorème est le résultat suivant.

\begin{corollary}
\textit{La solution }$u$\textit{\ n'appartient pas à l'espace }$%
BV([0,T^{\ast }[;\mathbf{\tilde{B}}_{\infty }^{-1,\infty }).$
\end{corollary}

\begin{proof}
En utilisant l'injection de $\mathbf{L}^{n}$ dans $\mathbf{\tilde{B}}%
_{\infty }^{-1,\infty }$, Le Théorème 1 nous permet de construire par ré%
currence une suite strictement croissante $(t_{j})_{j\in \mathbb{N}}$ d'é%
lements de $]0,T^{\ast }[$ telle que, pour tout $j\in \mathbb{N},$ $%
\left\Vert u(t_{j+1})-u(t_{j})\right\Vert _{\mathbf{B}_{\infty }^{-1,\infty
}}\geq \varepsilon _{\ast }.$ D'où le résultat.
\end{proof}

Avant de passer à la démonstration du Théorème 1, nous énonçons quelques ré%
sultats préliminaires. Pour les démonstrations de ces résultats ainsi que
pour les définitions de la décomposition de Littlewood-Paley, du paraproduit
de Bony et des espaces de Besov, nous renvoyons les lecteurs aux références
\cite{Can} et \cite{Lem}.\medskip

Le premier résultat est une version améliorée du théorème d'existence de
Kato \cite{Kat}.

\begin{theorem}[Théorème de Kato]
\textit{Soit} $v_{0}\in \mathbf{L}_{\sigma }^{n}.$ \textit{Il existe un
unique temps maximal} $T_{\ast }\overset{\text{déf}}{=}T_{K}^{\ast
}(v_{0})\in ]0,\infty ]$ \textit{et une unique fonction vectorielle }$v%
\overset{\text{déf}}{=}S_{K}^{\ast }(v_{0})\in \cap _{0<T<T_{\ast }}\mathbf{L%
}_{K}^{n}(Q_{T})$ \textit{solution maximale sur }$]0,T_{\ast }[$\textit{\
des équations }(NSI)\textit{\ associées à la donnée initiale }$v_{0},$%
\textit{\ où }$\mathbf{L}_{K}^{n}(Q_{T})$ \textit{est l'espace des fonctions}
$w\in C([0,T];\mathbf{L}_{\sigma }^{n})$ \textit{telles que} $\sqrt{t}w\in
C([0,T];\mathbf{L}^{\infty })$\textit{\ et} $\lim_{t\rightarrow 0}\sqrt{t}%
\left\Vert w(t)\right\Vert _{\mathbf{\infty }}=0.$\textit{\ Cette solution }$%
v$ \textit{est de classe }$C^{\infty }$\textit{\ sur l'ouvert }$Q_{T_{\ast
}},$\textit{\ plus précisément }$v\in \cap _{j,i\in \mathbb{N}%
}C_{t}^{i}(]0,T_{\ast }[,\mathbf{\tilde{B}}_{\infty }^{j,\infty }).$ \textit{%
Enfin, il existe une constante positive }$\varepsilon _{n},$\textit{\ qui ne
dépend que de }$n,$\textit{\ telle que si pour un réel positif }$T$\textit{\
on a} $\left( 1+\left\Vert v_{0}\right\Vert _{\mathbf{n}}\right) \sup_{0<t<T}%
\sqrt{t}\left\Vert e^{t\Delta }v_{0}\right\Vert _{\mathbf{\infty }}\leq
\varepsilon _{n}$ \textit{alors }$T_{K}^{\ast }(v_{0})\geq \inf (1,T).$
\end{theorem}

Une conséquence directe de ce théorème est le lemme principal suivant.

\begin{lemma}
\label{Lem1}\textit{Soit }$v_{0}\in \mathbf{L}_{\sigma }^{n}.$ \textit{On
pose }$v=S_{K}^{\ast }(v_{0})$\textit{\ et }$T_{K}^{\ast }=T_{K}^{\ast
}(v_{0}).$\textit{\ Alors pour tout }$t_{0}\in ]0,T_{K}^{\ast }[$\textit{\
on a }$T_{K}^{\ast }(v(t_{0}))=T_{K}^{\ast }-t_{0}$ \textit{et} $S_{K}^{\ast
}(v(t_{0}))=v(.+t_{0}).$\textit{\ Si on suppose que }$0<T_{K}^{\ast
}-t_{0}\leq 1$\textit{\ alors}
\begin{equation}
I_{\ast }(v_{0},t_{0})\overset{\text{déf}}{=}\left( 1+\left\Vert
v(t_{0})\right\Vert _{\mathbf{n}}\right) \sup_{0<t<T_{K}^{\ast
}(v_{0})-t_{0}}\sqrt{t}\left\Vert e^{t\Delta }\left( v(t_{0})\right)
\right\Vert _{\mathbf{\infty }}>\varepsilon _{n}.  \label{E3}
\end{equation}
\end{lemma}

Le lemme suivant caractérise l'effet régularisant de l'opérateur $\mathbb{L}%
. $

\begin{lemma}
\label{Lem2}\textit{Il existe une constante }$C_{n}>0$\textit{\ telle que
pour tous }$T\in ]0,1],$ $\alpha \in \{1,2\}$\textit{\ et} $r\in \mathbb{R},$
\textit{L'opérateur }$\mathbb{L},$\textit{\ défini par }(\ref{op}),\textit{\
est continu de }$C([0,T];\mathbf{\tilde{B}}_{\infty }^{r,\infty })$ \textit{%
dans }$C([0,T];\mathbf{\tilde{B}}_{\infty }^{r+\alpha ,\infty })$ \textit{et
sa norme est inférieure à }$C_{n}T^{\frac{2-\alpha }{2}}.$
\end{lemma}

Pour énoncer le dernier lemme, nous introduisons la définition d'une \textit{%
version affaiblie} du para-produit vectoriel de Bony. Soient $f$ et $g$ :$%
\mathbb{R}^{n}\rightarrow \mathbb{R}^{n}$ deux fonctions. On définit,
formellement, les deux opérateurs bilinéaires $\pi _{0}$ et $\pi _{1}$ par $%
\pi _{i}(f\otimes g)=\sum_{k=0}^{\infty }S_{k+i}(f)\otimes \Delta _{k}(g),$ $%
i=0,1.$ Rappelons que dans le cas où $f$ et $g$ sont dans $S(\mathbb{R}%
^{n})^{n}$ on a bien l'identité $f\otimes g=\pi _{0}(f\otimes g)+\pi
_{1}(g\otimes f).$

\begin{lemma}
\label{Lem3}\textit{Soit }$s>1.$\textit{\ Il existe une constante positive }$%
C_{n,s}$\textit{\ telle que les opérateurs }$\pi _{0}$\textit{\ et }$\pi
_{1} $\textit{\ sont continus de }$\mathbf{\tilde{B}}_{\infty }^{-1,\infty
}\times \mathbf{\tilde{B}}_{\infty }^{1+s,\infty }$ (\textit{resp. }$\mathbf{%
C}_{0}(\mathbb{R}^{n})\times \mathbf{\tilde{B}}_{\infty }^{1+s,\infty })$
\textit{dans} $\mathbf{\tilde{B}}_{\infty }^{s,\infty }$ (\textit{resp.} $%
\mathbf{\tilde{B}}_{\infty }^{1+s,\infty })$ \textit{et leurs normes sont inf%
érieures à }$C_{n,s}.$ \textit{En particulier si }$f$\textit{\ et }$g$%
\textit{\ sont dans} $\mathbf{\tilde{B}}_{\infty }^{1+s,\infty }$ \textit{%
alors} $f\otimes g=\pi _{0}(f\otimes g)+\pi _{1}(g\otimes f)$ \textit{dans} $%
\mathbf{\tilde{B}}_{\infty }^{1+s,\infty }.$
\end{lemma}

\section{Démonstration du Théorème}

Nous partageons la démonstration en trois étapes.

\noindent 1$^{\grave{e}re}$ étape. Montrons que $T^{\ast }=T_{K}^{\ast
}(u_{0})$ et que $u=S_{K}^{\ast }(u_{0})$ (il en résulte, en particulier,
que $u\in C^{\infty }(Q_{T^{\ast }})).$ En vertu de l'unicité des solutions
des équations de Navier-Stokes dans $C([0,T];\mathbf{L}^{n})$ \cite{Fur},
nous avons $T^{\ast }\geq T_{K}^{\ast }(u_{0})$ et $u=S_{K}^{\ast }(u_{0})$
sur $[0,T_{K}^{\ast }(u_{0})[.$ On conclut alors dès qu'on prouve que $%
T^{\ast }\leq T_{K}^{\ast }(u_{0}).$ Supposons que $T_{K}^{\ast
}(u_{0})<T^{\ast }.$ Alors l'ensemble $S_{K}^{\ast }(u_{0})\left(
[0,T_{K}^{\ast }(u_{0})[\right) =u\left( [0,T_{K}^{\ast }(u_{0})[\right) $
est un précompact de $\mathbf{L}^{n}.$ Utilisant ensuite le fait que pour
toute $f\in \mathbf{L}^{n},$ sup$_{0<s<1}\sqrt{s}\left\Vert e^{s\Delta
}f\right\Vert _{\infty }\leq c_{n}\left\Vert f\right\Vert _{n}$ et $%
\lim_{s\rightarrow 0}\sqrt{s}\left\Vert e^{s\Delta }f\right\Vert _{\infty
}=0,$ on montre qu'il existe $\lambda \in ]0,1[$ tel que, pour tout $%
t_{0}\in \lbrack 0,T_{K}^{\ast }(u_{0})[,$ on a
\begin{equation*}
\left( 1+\left\Vert S_{K}^{\ast }(u_{0})(t_{0})\right\Vert _{\mathbf{n}%
}\right) \sup_{0<t<\lambda }\sqrt{t}\left\Vert e^{t\Delta }S_{K}^{\ast
}(u_{0})(t_{0})\right\Vert _{\mathbf{\infty }}\leq \varepsilon _{n}.
\end{equation*}%
Prenons $t_{0}$ tel que $0<T_{K}^{\ast }(u_{0})-t_{0}<\lambda ,$ il vient $%
I_{\ast }(u_{0},t_{0})\leq \varepsilon _{n},$ ce qui est absurde d'après (%
\ref{E3}).

\noindent 2$^{\grave{e}me}$ étape. Montrons que pour tout $0<a<T^{\ast },$ $%
u\notin L^{\infty }([a,T^{\ast }[,L^{\infty }).$ On raisonne par l'absurde
et on pose $v_{0}=u(a)$ et $v=S_{K}^{\ast }(v_{0}).$ Alors, d'après le Lemme %
\ref{Lem1} et l'étape précédente, on a $M\overset{\text{déf}}{=}%
\sup_{0<t<T_{K}^{\ast }(v_{0})}\left\Vert v(t)\right\Vert _{\infty }<\infty
. $ Un calcul élémentaire utilisant le lemme de Gronwall, l'inégalité de
Young et le fait que le noyau de l'opérateur $e^{t\Delta }\mathbb{P}\nabla $
est dans $L^{1}(\mathbb{R}^{n})^{n\times n}$ et que sa norme est inférieure à
$\frac{C}{\sqrt{t}},$ nous permet de prouver que $N\overset{\text{déf}}{=}%
\sup_{0<t<T_{K}^{\ast }(v_{0})}\left\Vert v(t)\right\Vert _{n}$ est fini$.$
Par conséquent, pour tout $t_{0}\in \lbrack 0,T_{K}^{\ast }(v_{0})[,$ on a
\begin{equation*}
I_{\ast }(v_{0},t_{0})\leq C(1+N)\,M\sqrt{T_{K}^{\ast }(v_{0})-t_{0}},
\end{equation*}%
ce qui contredit (\ref{E3}).

\noindent 3$^{\grave{e}me}$ étape. Soit $\varepsilon >0$ pour lequel on
suppose qu'il existe $\omega \in S(\mathbb{R}^{n})^{n}$ vérifiant
\begin{equation*}
\overline{\lim_{t\rightarrow T^{\ast }}}\left\Vert u(t)-\omega \right\Vert _{%
\mathbf{B}_{\infty }^{-1,\infty }}<\varepsilon .
\end{equation*}%
Il existe alors $\delta _{0}\in ]0,T^{\ast }[$ tel que pour tout $t\in
\lbrack T^{\ast }-\delta ,T^{\ast }[$ on a $\left\Vert u(t)-\omega
\right\Vert _{\mathbf{B}_{\infty }^{-1,\infty }}<\varepsilon .$ Soit $\delta
\in ]0,\delta _{0}[$ un réel à fixer ultérieurement. On pose $%
w_{0}=u(T^{\ast }-\delta ),$ alors, d'après la 1$^{\grave{e}re}$ étape et le
Lemme \ref{Lem1}, $T_{K}^{\ast }(w_{0})=\delta $ et $w=S_{K}^{\ast
}(w_{0})=u(.+$ $T^{\ast }-\delta ).$ Par conséquent $\sup_{0<t<\delta
}\left\Vert w(t)-\omega \right\Vert _{\mathbf{B}_{\infty }^{-1,\infty
}}<\varepsilon .$ Soit $s>0.$ Le Théorème de Kato assure que $w\in
C([0,\delta ];\mathbf{\tilde{B}}_{\infty }^{s+1,\infty }),$ il en résulte,
d'après le Lemme \ref{Lem3},
\begin{equation*}
w(t)=e^{t\Delta }w_{0}+\sum_{j=0}^{1}\mathbb{L}\left( \mathbb{P}\nabla .\pi
_{j}\left[ (w-\omega )\otimes w\right] \right) +\mathbb{L}\left( \mathbb{P}%
\nabla .\pi _{j}\left[ \omega \otimes w\right] \right) (t).
\end{equation*}%
Utilisons encore le Lemme \ref{Lem3} et le fait que $\mathbb{P}\nabla $ est
continue de $\mathbf{\tilde{B}}_{\infty }^{r,\infty }$ dans $\mathbf{\tilde{B%
}}_{\infty }^{r-1,\infty }(r\in \mathbb{R}),$ on trouve que, pour tout $%
\delta ^{\prime }\in ]0,\delta \lbrack ,$ on a
\begin{equation*}
\sup_{0<t<\delta ^{\prime }}\left\Vert w(t)\right\Vert _{\mathbf{\tilde{B}}%
_{\infty }^{s+1,\infty }}\leq \left\Vert w_{0}\right\Vert _{\mathbf{\tilde{B}%
}_{\infty }^{s+1,\infty }}+C\{\varepsilon +\left\Vert \omega \right\Vert
_{\infty }\sqrt{\delta }\}\sup_{0<t<\delta ^{\prime }}\left\Vert
w(t)\right\Vert _{\mathbf{\tilde{B}}_{\infty }^{s+1,\infty }},
\end{equation*}%
où $C$ est une constante positive qui ne dépend que de $n$ et $s.$ Supposons
par l'absurde que $\varepsilon <\varepsilon _{\ast }\overset{\text{déf}}{=}%
\frac{1}{4C}.$ Choisissons, maintenant, $\delta $ tel que la quantité $%
C\{\varepsilon +\left\Vert \omega \right\Vert _{\infty }\sqrt{\delta }\}$
soit inférieure à $\frac{1}{2},$ on obtient que le $\sup_{0<t<\delta
}\left\Vert w(t)\right\Vert _{\mathbf{\tilde{B}}_{\infty }^{s+1,\infty }}$
est majoré par $2\left\Vert w_{0}\right\Vert _{\mathbf{\tilde{B}}_{\infty
}^{s+1,\infty }}.$ Rappelons que $\mathbf{\tilde{B}}_{\infty }^{s+1,\infty }$
s'injecte dans $\mathbf{L}^{\infty }$ et que $w=u(.+$ $T^{\ast }-\delta ),$
il vient que $\sup_{T^{\ast }-\delta <t<T^{\ast }}\left\Vert u(t)\right\Vert
_{\infty }$ est fini, ce qui est impossible d'après l'étape précédente.
Donc, pour toute $\omega \in S(\mathbb{R}^{n})^{n},$%
\begin{equation*}
\overline{\text{ }\lim_{t\rightarrow T^{\ast }}}\left\Vert u(t)-\omega
\right\Vert _{\mathbf{B}_{\infty }^{-1,\infty }}\geq \varepsilon _{\ast }.
\end{equation*}%
Par densité, cette dernière estimation reste vraie pour toute $\omega \in
\mathbf{\tilde{B}}_{\infty }^{-1,\infty }.$ \bigskip
\providecommand{\bysame}{\leavevmode\hbox
to3em{\hrulefill}\thinspace}



\providecommand{\bysame}{\leavevmode\hbox
to3em{\hrulefill}\thinspace} \providecommand{\MR}{\relax\ifhmode\unskip%
\space\fi MR }
\providecommand{\MRhref}[2]{
\href{http://www.ams.org/mathscinet-getitem?mr=#1}{#2} } \providecommand{%
\href}[2]{#2}

\end{document}